\documentclass{article}
\usepackage{graphicx}
\usepackage{amsmath,amsthm,amssymb}
\newtheorem{thm}{Theorem}
\newtheorem{lemma}{Lemma}
\newtheorem{corollary}{Corollary}

\begin{document}

\title{The Structure of Stable Vector Fields on Surfaces}
\author{John Berman, Sergei Bernstein}

\maketitle

\begin{abstract}
The Poincare-Hopf theorem tells us that given a smooth, structurally
stable vector field on a surface of genus $g$, the number of saddles
is $2-2g$ less than the number of sinks and sources. We
generalize this result by introducing a more complex combinatorial invariant. Using this tool, we demonstrate that
many such structurally stable vector fields are equivalent up to a set
of basic operations. We show in particular that for the sphere,
all such vector fields are equivalent. 
\end{abstract}

\section{Introduction}
In order to study combinatorial aspects of vector fields, we restrict our attention to surfaces that are compact, connected, and orientable. The Poincare-Hopf Theorem yields a simple combinatorial invariant of a vector field on such a surface; in particular, it tells us that $$U+I-A=2-2g$$ where we write $U, I, A$ for the numbers of sources, sinks, and saddles, respectively, and $g$ for the number of holes in $X$, also known as the genus of $X$. 

We seek to extend this simple result to a more powerful combinatorial invariant in the form of a graph embedded on $X$, with edges representing integral curves of the vector field. In order to accomplish this, we must restrict our attention to structurally stable vector fields, or vector fields whose geometric behavior is unchanged by small perturbations.

We show that our graph invariant is a triangulation of $X$, and, using this machinery, we seek to describe classes of vector fields that may be deformed into each other via simple, local operations we term vector field cobordisms. In our main result, we show that all structurally stable vector fields on the sphere are cobordant in this sense, and we discuss the extent to which this holds true for surfaces of higher genus. Unfortunately, for surfaces of higher genus, the existence of periodic integral curves hampers the project, so we leave the readers with some open questions.

\section{The Graph Invariant}

\subsection{Primitive and Generic Vector Fields}
Consider a smooth, structurally stable vector field on a (compact, connected, orientable) surface. By Peixoto's theorem \cite{P}, structural stability is equivalent to a simple set of properties, as follows. 

\begin{itemize}

\item The set of non-wandering points consists only of periodic orbits
  and fixed points. In other words, an integral curve becomes either a periodic orbit or a fixed point in the limit.

\item The set of fixed points is finite and consists only of
  hyperbolic equilibrium points. In particular, there are saddles of degree $4$, sinks, and sources. 

\item There are a finite number of attracting or repelling periodic orbits.

\item There are no saddle-saddle connections.

\end{itemize}

We shall call such a vector field {\it generic}, and in particular a
generic vector field with no such periodic orbits {\it primitive}. The
fixed points consist of sinks, sources, and saddles, the numbers of
which we will denote throughout $I$, $U$, and $A$, respectively. 

\subsection{The Invariant}
Suppose we have a primitive vector field $v$ on a surface $X$. As it
has a finite number of fixed points, it is  natural to construct a
graph $G(v)$ (or just $G$ if $v$ is known from context) with the
vertices corresponding to the fixed points of $v$. We then say that
there is a directed edge from $A$ to $B$ for each homotopy class of
integral curves that start at $A$ and end at $B$. We notice a few
properties 

\begin{itemize}

\item There is a natural way to pick an isotopy class of embeddings of $G$ on $X$.

\item By structural stability of $v$, $G$ must be 3-colorable with the
  colors corresponding to sources, sinks, and saddles.

\item Also by structural stability, each saddle point must have in-degree 2
  and out-degree 2.

\end{itemize}

This graph $G(v)$ will be our primary tool for studying vector
fields. 

If we remove all the sources and associated edges from $G$, we get a
new graph, which we denote $G_I(v)$ (or just $G_I$). We notice that in
$G_I$, every saddle connects to exactly $2$ sinks. Therefore, we can
construct yet another (undirected) graph $\tilde{G}_I(v)$ where the
vertices are the sinks of $G_I$, and there is an edge between two (not
necessarily distinct) sinks for each saddle that is connected to both
of those sinks. We can construct similar graphs $G_U(v)$ and
$\tilde{G}_U(v)$ by interchanging sinks and sources. It is immediately
clear by definition that $\tilde{G}_I$ consists of $I$ vertices and
$A$ edges, and that $\tilde{G}_U$ consists of $U$ vertices and
$A$ edges. Furthermore, all of these graphs may be naturally embedded on $X$, as $G$ itself may be embedded on $X$.  

We now present some examples in Figures \ref{fig:SphereExample} through \ref{fig:SinkGraph}. Note that we represent vector fields by their flows. 

\begin{figure}
  \centering
  \includegraphics[width=40mm,height=40mm]{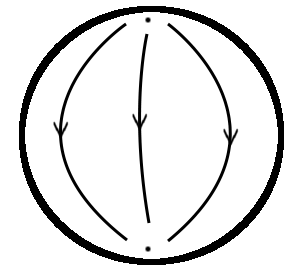}
  \caption{An example of a primitive vector field on a sphere}
  \label{fig:SphereExample}
\end{figure}

For the vector field in Figure \ref{fig:SphereExample}, $G$ consists of a source, a sink, and a
directed edge from the source to the sink. The graphs $G_U$ and $\tilde{G}_U$ are each just a single vertex.

\begin{figure}[h!]
  \centering
  \includegraphics[width=40mm]{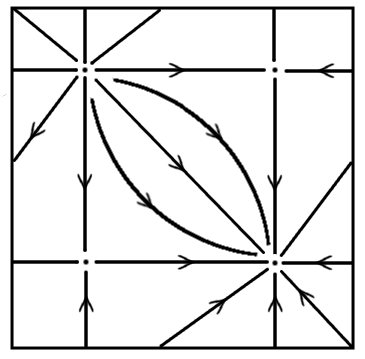}
  \caption{An example of a primitive vector field on a torus}
  \label{fig:TorusExample}
\end{figure}

\pagebreak

\begin{figure}
  \centering
  \includegraphics[width=40mm]{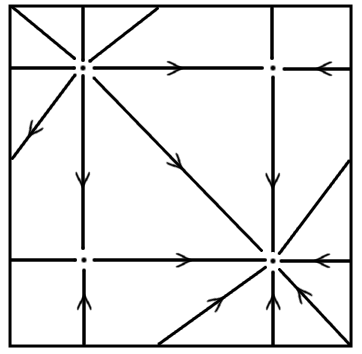}
  \caption{$G$ in the previous example}
  \label{fig:TorusExampleGraph}
\end{figure}

\begin{figure}
  \centering
  \includegraphics[width=40mm]{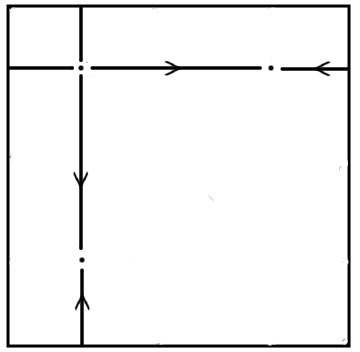}
  \caption{$G_U$ in the previous example}
  \label{fig:TorusExampleSaddle}
\end{figure}

\begin{figure}[h!]
  \centering
  \includegraphics[width=40mm]{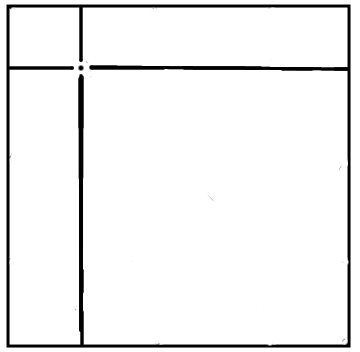}
  \caption{$\tilde{G}_U$ in the previous example}
  \label{fig:TorusExampleSource}
\end{figure}

In Figure \ref{fig:TorusExample}, we present an example of a primitive vector field, $v_T$, on a torus. In this figure, we draw the torus as a square whose opposite edges are glued together. Figure \ref{fig:TorusExampleGraph} shows a natural embedding of $G(v_T)$ on the torus. Figures \ref{fig:TorusExampleSaddle} and \ref{fig:TorusExampleSource} are embeddings of $G_U(v_T)$ and $\tilde{G}_U(v_T)$, respectively.

\pagebreak

\begin{figure}
  \centering
  \includegraphics[width=60mm]{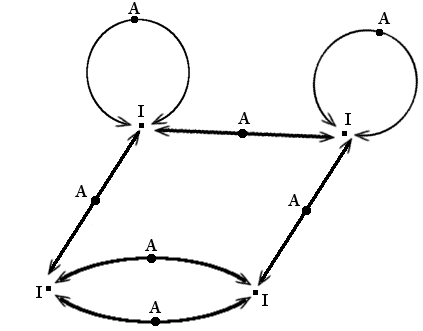}
  \caption{A non-trivial example of $G_I$ the sphere (stereographically projected to the plane)}
  \label{fig:SaddleGraph}
\end{figure}

\begin{figure}
  \centering
  \includegraphics[width=60mm]{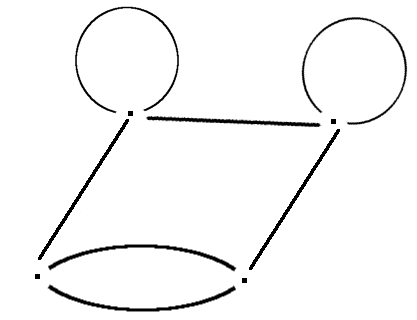}
  \caption{$\tilde{G}_I$ in the previous example}
  \label{fig:SinkGraph}
\end{figure}

In Figure \ref{fig:SaddleGraph}, we give an example of a possible value of $G_I(v)$ on the sphere. For the sake of clarity, we have stereographically projected the sphere onto the plane. Figure \ref{fig:SinkGraph} then gives us a possible value of $\tilde{G}_I(v)$ on the sphere. 

\pagebreak

\subsection{Conditions on $G(v)$}

Given an embegging of a graph into $X$, we define a region to be a connected component of the complement of that graph in $X$.

\begin{thm}
Given an embedding of the graph $\tilde{G}_I$ ($\tilde{G}_U$) in $X$,
\begin{itemize}

\item In each region, there is exactly one source (sink) when the vertices of $G$ are added.

\item Each region (taken as an open set) is diffeomorphic to a open
  set of $\mathbb{R}^2$. 

\item The graph $\tilde{G}_I$ is connected.

\end{itemize}
\label{thm:GConnect}
\end{thm}

\begin{proof} 
Each region must contain at
least one source, as each region has some sink on its boundary which
must connect to a source in the region via an integral curve. Say that
there are $v_I$, $e_I$, $f_I$ vertices, edges, faces respectively in
$\tilde{G}_I$. This means that $v_I=I$, $e_I=A$, and $f_I\le U$, with
equality if and only if there is exactly one source of $G$ in each region. Letting $g$ denote
the genus of $X$, we can write the Poincare-Hopf result
as $$I+U=A+2-2g$$ and Euler's result as $$v_I+f_I\ge e_I+2-2g$$
with equality only when each region is diffeomorphic to an open set of $\mathbb{R}^2$ and $\tilde{G}_I$ is connected. Combining our observations results in $$2-2g\le v_I+f_I-e_I\le U+I-A=2-2g$$ Thus, both equalities cases hold, proving the theorem. Of course the same argument can be made for
$\tilde{G}_U$.
\end{proof}

From the first point of Theorem \ref{thm:GConnect} and the observation that saddles
have in degree 2 and out degree 2, it follows immediately that
$\tilde{G}_I$ and $\tilde{G}_U$ are dual graphs on $X$. 

Suppose that $v$ has at least one saddle. Let us look at a single
region in $G_I(v)$. As $v$ contains a saddle, $G_I$ contains an
edge. Suppose that the region is bounded by the edges $e_1, \ldots,
e_n$ in order (edges may repeat, but would then be oriented
differently). Let us examine an edge $e_i$. Then it is not hard to see that both vertices of $e_i$ must define a homotopy class of integral curves
(with the correct orientation) connecting to the single source in the
region, for they have nowhere else to connect. This gives us a triangle between the single source and these two vertices. Hence, when $G$ is
reconstructed, the particular region of $G_I$ is partitioned into
triangles. It follows that $G$ is a triangulation of the entire
manifold.

\begin{corollary}
As long as $v$ has at least one
saddle, $G$ is a triangulation of $X$. Note that by the word
triangulation, we dictate not only that each region be triangular, but
also that each region be diffeomorphic to an open set of
$\mathbb{R}^2$, as in Theorem \ref{thm:GConnect}.
\label{cor:triangulation} 
\end{corollary}

Of course, $G$ will also satisfy some additional trivial properties. As $v$ is structurally stable, with no saddle-saddle connections, $G$ will be $3$-colorable by the ``colors'': sink, source, and saddle. In fact, each sink will have out-degree 0, each source in-degree 0, and each saddle both in-degree and out-degree 2. Call a graph that satisfies all of these properties {\it saddled}. It happens that $G$ is characterized by being a saddled triangulation.

\begin{thm}
Given a graph $G$ on a surface $X$, there is a primitive vector field $v$ satisfying $G=G(v)$ if and only if $G$ is a saddled triangulation of $X$.
\end{thm}

\begin{proof} Corollary \ref{cor:triangulation} proves the forward direction. We shall now show the other direction. Suppose $G$ is a
saddled triangulation of $X$. We can embed $G$ into $X$ in such a way
that all the edges are smooth curves. This splits the surface into some triangular regions. We note that because $G$ is a saddled triangulation, the vertices of each triangle must be a source, a sink, and a saddle. It is sufficient to specify the flow in each triangular region. We can just specify the flow to be diffeomorphic to the one in Figure \ref{fig:Triangle}.
\end{proof}

\begin{figure}
  \centering
  \includegraphics[width=40mm]{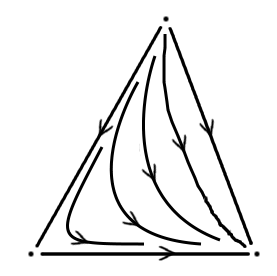}
  \caption{An example of a smooth flow within a triangular region}
  \label{fig:Triangle}
\end{figure}

\subsection{Non-primitive vector fields}

Suppose now that $v$ does have periodic integral curves. We can think of each periodic curve as dividing a small open set around it into two regions. Consider its behavior in one of the regions. The curve acts either as a `sink' or a `source,' with integral curves either progressing away from it or towards it in its neighborhood. In fact, within this region, we can close off the integral curve (or associate it with a single point) in such a way that it truly does define either a sink or a source (and no periodic curve) (see Figure \ref{fig:PSink}).

Stated more clearly, suppose there exist periodic integral curves. Ennumerate the regions $R_1, \ldots, R_n$ into which they divide $X$. In each region, the bounding integral curves may be associated with (closed off to yield) points, with resulting manifolds $X_1, \ldots, X_n$ and derived vector fields $v_1, \ldots, v_n$. Each of the derived vector fields exist on the corresponding manifold and are equivalent to $v$ except that the original periodic curves are now associated with either sinks or sources. In particular, this means that each $v_i$ is primitive, and therefore the results of the previous sections hold.

\begin{figure}
  \centering
  \includegraphics[width=100mm]{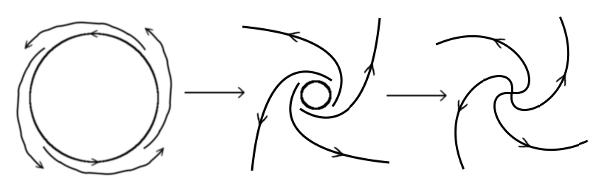}
  \caption{A picture to clarify why a periodic curve can be viewed as a source (or a sink)}
  \label{fig:PSink}
\end{figure}

\section{Vector Field Cobordisms}

\subsection{Definitions}

The object of this section is to describe elementary local transformations of vector fields, which we  call {\it vector field cobordisms}. We define two classes of cobordisms for primitive vector fields. Given two sources, each connecting to a given saddle, we can smoothly deform the vector field in such a way that the sources and the saddle are `pushed' together. Outside of a small neighborhood of the three points, the entire neighborhood acts as a source (see Figure \ref{fig:CobordismProcess}), with the vector field emanating away from the region. Thus we define a cobordism to be the transformation from this vector field to another where the three points have been identified with a single source. Of course, the reverse transformation is also a cobordism. The second class of cobordisms we define takes two sinks and a saddle to a sink in an identical way. We call these source-cobordisms and sink-cobordisms, respectively.

We say that two vector fields are {\it cobordant} if there is a series of cobordisms and isotopies that takes one to the other. This gives us {\it cobordism classes} of vector fields.

\begin{figure}
  \centering
  \includegraphics[width=80mm]{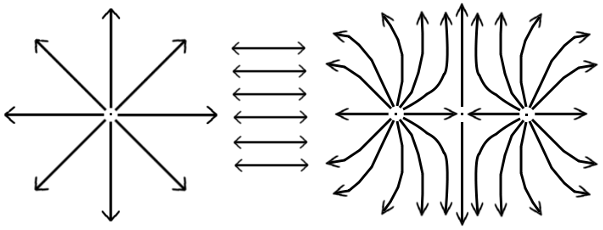}
  \caption{A source cobordism}
  \label{fig:CobordismProcess}
\end{figure}

\subsection{Primitive Vector Fields}

Consider the effect of cobordisms on the graph $\tilde{G}_I(v)$. Sink-cobordisms act on two adjacent (and distinct!) vertices and an edge between them by removing the edge and equating the two vertices, so that all the edges of each vertex are now assigned to the new vertex. Source-cobordisms take an edge between two (distinct!) regions and remove it, equating the regions. In this way, cobordisms are not only operations on vector fields, but also on connected graphs embedded on $X$.

\begin{lemma}
Every primitive vector field can be reduced via cobordisms to a vector field with just one source and one sink.
\label{lem:OneGraph}
\end{lemma}

\begin{proof}
Suppose a vector field with more sinks (for example) could not be further reduced via cobordisms. But $\tilde{G}_I$ is connected by Theorem 1, so there are two adjacent distinct vertices with an edge between them, and we can reduce via a sink-cobordism. On the other hand, all primitive vector fields have infinitely many integral curves that need a place to start and a place to end, so every primitive vector field must contain at least one sink and one source. 
\end{proof}

Call such a primitive vector field with just one source and one sink {\it reduced}. To what extent can these differ? The Poincare-Hopf theorem tells us that the number of saddles is specified at $2g$. Each saddle must connect twice to the source and twice to the sink. However, we know that $G$ is a triangulation of $X$, so we see that the number of homotopy classes of integral curves between the sink and the source is thus fixed at $4g$ (provided $g>0$). Thus, given the manifold's genus $g$, there is only one possible graph $G(w)$ for $w$ reduced. 
Care must be taken, however, as $G$ may embed in multiple ways into $X$. In fact, if we assign a certain embedding to the identity element in the mapping class group of $X$, then the different embeddings of $G$ correspond to elements of the mapping class group. There is not, therefore, necessarily only one cobordism class of primitive vector fields.

It is true, however, that if $X$ is a sphere, all primitive vector fields are cobordant. As above, all primitive vector fields are certainly cobordant to a reduced vector field with just one sink, one source, and one homotopy class of integral curves between them. As the mapping class group of the sphere is trivial, we can deform all such vector fields into each other, so there is just one cobordism class.

\begin{thm} 
There is just one cobordism class of primitive vector fields on the sphere.
\end{thm}

We now show that there is one cobordism class for the torus and we conjecture that there is also only one cobordism class for surfaces of higher genus.

\subsection{Primitive Vector Fields on a Torus}

Suppose we have two vector fields $v, w$ on $X$ such that the natural embeddings of $G(v)$ and $G(w)$ are isotopic. Because these graphs are triangulations, it is clear that there is an isotopy that takes $v$ to $w$. 

Next, we suppose that we have two vector fields $v, w$ such that the natural embeddings of $\tilde{G}_U(v)$ and $\tilde{G}_U(w)$ are the same. By Corollary \ref{cor:triangulation}, each region made by these graphs is diffeomorphic to $\mathbb{R}^2$, and there is exactly one sink in each of these regions. Because the mapping class group of $\mathbb{R}^2$ is trivial, there is only one way to put one sink in each region up to isotopy. In other words, we conclude that $G(v)$ is isotopic to $G(w)$, which in turn implies that $v$ is isotopic to $w$.

Call the complete tripartite graph with one source, one sink, and two saddles $G$.
By lemma \ref{lem:OneGraph}, we know that every vector field is cobordant to a vector field $v$ such that $G(v) = G$ as a graph. $\tilde{G}_U$ is just the graph with one vertex and two edges. Therefore, to show that every vector field on the torus is cobordant, it is sufficient to show that any embedding of $\tilde{G}_U$ is cobordant to every other embedding. In particular it is sufficient to show that the simple embedding in Figure \ref{fig:SimpleExample} is cobordant to every other embedding. Suppose we have an embedding $E$ that is not isotopic to this one. There is some element of the mapping class group of the torus that takes our simple embedding to $E$. If two embeddings correspond to the same element of the mapping class group, there is an isotopy taking one to the other. Therefore, we want to show that we can use cobordisms to reach an embedding corresponding to any element of the mapping class group. 

\begin{figure}
  \centering
  \includegraphics[width=40mm]{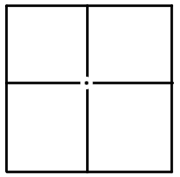}
  \caption{A simple embedding of $\tilde{G}_U$: Note, we represent the torus as a square with opposite edges set equal}
  \label{fig:SimpleExample}
\end{figure}

However, we can perform the series of cobordisms shown in Figure \ref{fig:DehnTwist}. In this way, we have found two embeddings that are not isotopic but are cobordant. In fact, we notice that Figure \ref{fig:DehnTwist} simply gives us a Dehn twist on the torus along one of the two edges $e_1$ and $e_2$ of the graph. Therefore, we want to show that we can reach every element of the mapping class group by performing Dehn twists along these two edges. Unfortunately, we have a slight complication. As we perform these Dehn twists, the edges along which we perform the twists move. To deal with this problem, we draw two curves $C_1$ and $C_2$. We draw $C_1$ (resp. $C_2$) where $e_1$ (resp. $e_2$) starts out in Figure \ref{fig:SimpleExample}. As we perform the Dehn twists around $e_1$ and $e_2$, we move both $e_1$ and $e_2$ around, but we leave $C_1$ and $C_2$ stationary. Now, we also use $g_1$ and $g_2$, the Dehn twists around $C_1$ and $C_2$. Note that $g_1$ and $g_2$ are elements of the mapping class group of the torus. Suppose that through Dehn twists around $e_1$ and $e_2$ we have obtained an element $x$ of the mapping class group of the torus. How can we describe another Dehn twist around $e_1$. Because $e_1$ has been moved by $x$, we must have that a Dehn twist around $e_1$ acts like the element $xg_1x^{-1}$ of the mapping class group. Applying this to $x$, we get $xg_1$. Thus, we see that performing a Dehn twist around $e_1$ is the same as multiplying on the right by $g_1$. The same can be said about $e_2$ and $g_2$. However, we know by the Lickorish Twist Theorem \cite{L} that $g_1$ and $g_2$ generate the mapping class group of the torus. It follows that we can get every element of the mapping class group by performing Dehn twists around $e_1$ and $e_2$. Therefore, we must have that the torus has only one cobordims class, as we desired. 

\begin{figure}
  \centering
  \includegraphics[width=100mm]{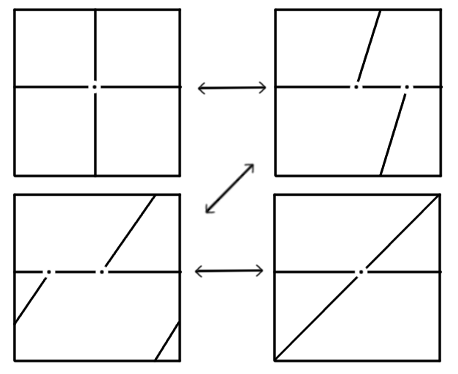}
  \caption{A series of cobordisms and isotopies}
  \label{fig:DehnTwist}
\end{figure}

\subsection{Non-Primitive Vector Fields}

When allowing for periodic curves, the two cobordisms we have defined will certainly not be sufficient to alter or remove periodic curves, thus creating an infinite class of cobordism classes. To account for this, we propose an additional three classes of cobordisms, each of which is in spirit a local collapsing of two objects, just like our previous cobordisms.

Consider a periodic curve which borders a region and acts as a source in that region. Consider also a point source in the region and a saddle which connects to both the point source and the periodic curve. Once again, we can continuously deform the two points together and into a neighborhood of the periodic curve, such that the periodic curve, when deformed slightly to encompass the two points, continues to act as a source. As above, we can then delete the point source and saddle and treat their integral curves as integral curves of the periodic curve. There is of course a symmetric cobordism with a sink-like periodic curve, a point sink, and a saddle.

These cobordisms provide a mechanism for handling periodic curves, but they still do not allow us to remove periodic curves, which is essential to avoid large sets of cobordism classes. Thus we propose a final cobordism, not quite in the spirit of the others. Given a region of genus 0 bordered by exactly one periodic curve which acts as a source (sink), and containing exactly one zero point, we propose that we can shrink the periodic curve to a point, swallowing the zero point inside, and leaving the remains of the periodic curve to act either as a sink or source depending on the original behavior of the periodic curve outside this region. This allows us to `swallow' periodic curves provided their interiors are as simple as possible (reduced, effectively). This operation is shown in Figure \ref{fig:PCCobordism}.

\begin{figure}
  \centering
  \includegraphics[width=80mm]{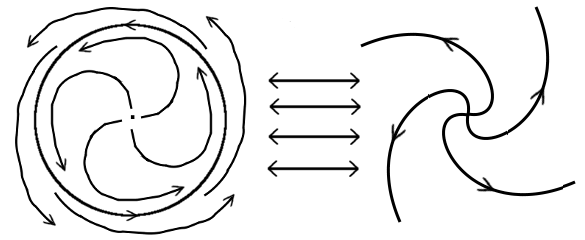}
  \caption{A basic periodic curve cobordism}
  \label{fig:PCCobordism}
\end{figure}

\begin{thm} All structurally stable vector fields on the sphere are cobordant to each other.
\end{thm}

\begin{proof} 
If there are no periodic curves, we are done as above. We will proceed by induction on the number of periodic curves, assuming there are $n$ periodic curves in the smallest counterexample to the theorem. Certainly, we can always find a region bounded by just one periodic curve. Treating this region itself as a sphere, we note that all cobordisms of the new sphere are proper cobordisms of the original sphere. Therefore we can reduce this to the vector field with just one sink and one source, one of which corresponds to the original periodic curve. Then we can use our final cobordism to remove the periodic curve, thus contradicting the induction hypothesis.
\end{proof}

\section{Conclusion}
Given an appropriate vector field $v$ on a surface $X$, we have developed the combinatorial tool of a graph $G$ on $X$ composed of homotopy classes of integral curves of $v$. By reducing the graph $G$ to just sinks or just sources and then rebuilding, we have shown that $G$ is a triangulation of $X$. Using this fact, combined with the tools of the sink and source-only graphs, we have demonstrated that vector fields on the sphere may be transformed into each other via basic localized operations which we call cobordisms.

Unfortunately, the presence of periodic integral curves keeps us from generalizing to higher genus surfaces. This motivates a number of open questions:

\begin{itemize}

\item Are all primitive vector fields on surfaces of higher genus cobordant?

\item How many cobordism classes of generic vector fields exist for surfaces of fixed genus with fixed periodic curve configurations?

\item How many cobordism classes of generic vector fields exist for surfaces of fixed genus with fixed numbers of periodic curves?

\item What can be said about vector fields which are not structurally stable?

\end{itemize}

\section{Acknowledgements}
The authors would like to acknowledge the generous support of Professor Victor Guillemin, who gave the initial idea for this research and has provided direction for us many times. We would also like to thank John Reed, chairman of the MIT corporation, for funding this research. Finally, we would like to thank the MIT Undergraduate Research Opportunities Program for making undergraduate research like this possible.


\begin{thebibliography}{9}
\bibitem{L} Lickorish, W.B.R. A finite set of generators for the homeotopy group of a $2$-manifold. {\it Proc. Cambridge Philos. Soc.} {\bf 60} (1964) 769-778
\bibitem{P} Peixoto, M. M. Structural Stability on two-dimensional manifolds. {\it Topology} {\bf 1} (1962) 101-120
\end{thebibliography}
\end{document}